\documentclass[12 pt, a4paper]{article}

\usepackage{amsfonts}

\usepackage{epsfig}
\usepackage{enumitem}
\usepackage{psfrag}

\newcommand{\fracs}[2]{{ \textstyle \frac{#1}{#2} }}
\newtheorem{theorem}{Theorem}[section]
\newtheorem{conjecture}{Conjecture}[section]
\newtheorem{corollary}{Corollary}[section]

\newtheorem{remark}{Remark}[section]
\newenvironment{proof}{
 \bgroup\noindent\small{\bf Proof\ }}{
 \nolinebreak\hbox{\ $\Box$}
 \egroup}

\newcommand{\dA}{\, {\rm d} A}
\newcommand{\dW}{\, {\rm d} W}
\newcommand{\dM}{\, {\rm d} M}
\newcommand{\dt}{\, {\rm d} t}
\newcommand{\ds}{\, {\rm d} s}
\newcommand{\dX}{\, {\rm d} X}

\newcommand{\dv}{\, {\rm d} v}

\newcommand{\e}{\, {\rm e}}

\newcommand{\eps}{\varepsilon}

\newcommand{\Ito}{It{\^o} }

\newcommand{\hL}{\widehat{L}}
\newcommand{\hP}{\widehat{P}}
\newcommand{\hY}{\widehat{Y}}

\newcommand{\hU}{\widehat{U}}

\newcommand{\EE}{\mathbb{E}}
\newcommand{\VV}{\mathbb{V}}
\newcommand{\RR}{\mathbb{R}}

\begin{document}

\title{Stochastic finite differences and multilevel Monte Carlo for a class of SPDEs in finance}
\author{Michael B.~Giles and Christoph Reisinger\\[0.1in]
        Oxford-Man Institute of Quantitative Finance\\
        and Mathematical Institute, University of Oxford}

\maketitle

\begin{abstract}
 In this article, we propose a Milstein finite
difference scheme for a stochastic partial differential equation (SPDE)
describing a large particle system.
We show, by means of Fourier analysis, that the discretisation on an 
unbounded domain is convergent of first 
order in the timestep and second order in the spatial grid size, and that
the discretisation is stable with respect to boundary data.
Numerical experiments clearly
indicate that the same convergence order also holds for boundary-value problems.
Multilevel path simulation, previously used for 
SDEs, is shown to give substantial complexity gains compared to a standard 
discretisation of the SPDE or direct simulation
of the particle system.
We derive complexity bounds and illustrate the results by an application to 
basket credit derivatives.
\end{abstract}

\section{Introduction}


Various stochastic partial differential equations (SPDEs) have emerged over the last two decades in different areas of mathematical finance.
A classical example is the Heath-Jarrow-Morton interest rate model \cite{heath92} of the form
\begin{eqnarray}
\label{hjm}
{\rm d} r(x,t) &=& \frac{\partial}{\partial x}\left(r(t,x) + 
\frac{1}{2} \left| \textstyle\int_t^{t+x}\sigma(t,u){\,\rm d}u\right|^2\right) {\,\rm d}t + \sigma(t,t+x) {\,\rm d}W_t,
\end{eqnarray}
where $r(x,t)$ is the forward rate of tenor $x$ at time $t$ and $\sigma(t,t+x)$ its instantaneous
volatility.
In \cite{bk08},
a similar equation has been proposed more recently to model electricity forwards.
Most of the SPDEs studied share with (\ref{hjm}) the property that the derivatives of the solution only appear in the drift term; in the case of (\ref{hjm}) the volatility of the Brownian driver does not depend on the solution $r$ at all.
Numerical methods for hyperbolic SPDEs of the type (\ref{hjm}) have been
studied, for example, in \cite{roth02}.

This article, in contrast, considers the parabolic SPDE
\begin{equation}
\label{spde}
\dv = -\mu\, \frac{\partial v}{\partial x} \dt 
+ \frac{1}{2} \, \frac{\partial^2  v }{\partial x^2} \dt
- \sqrt{\rho}\, \frac{\partial v}{\partial x} \dM_t,
\end{equation}
where $M$ is a standard Brownian motion, and $\mu$ and $0\le \rho\le 1$ are real-valued parameters.
It is clear that the behaviour of this equation is fundamentally different from those with additive or multiplicative noise.

The significance of (\ref{spde}) for the following applications is that it describes the limiting density of a large system of exchangeable particles.
Specifically, if we consider the system of SDEs
\begin{equation}
\label{sdesys}
\dX_t^i = \mu \dt + \sqrt{1 \!-\! \rho}\, \dW_t^i + \sqrt{\rho}\, \dM_t, 
\end{equation}
for $1\le i \le N$, with $\langle \dW_t^i, \dW_t^j\rangle = \delta_{ij}$ and
$\langle \dW_t^i, \dM_t\rangle = 0$, where 
$X_0^i$ are assumed i.i.d.\ with finite second moment,
the empirical measure
\[
\nu_t^N = \frac{1}{N} \sum_{i=1}^N \delta(\cdot-X_t^i)
\]
has a limit $\nu_t$ for $N\rightarrow \infty$,
whose density $v$ satisfies (\ref{spde}) in a weak sense. For a derivation
of this result in the more general context of quasi-linear PDEs see \cite{kx99}. 
While the motivation in \cite{kx99} is to use a large particle system (\ref{sdesys}) to approximate the solution to the SPDE (\ref{spde}),
our view point is to use (\ref{spde}) as an approximate model for a large particle system,
and we will argue later the (computational) advantages of this approach in situations when
the number of particles is large.

As a first possible application, one may consider $X^i$ as the log price processes of a basket of equities, which have idiosyncratic components $W^i$ and share a common driver $M$ (the ``market''). If the size of the basket is large enough, the solution to the  SPDE can be used to find the values of basket derivatives. In this paper, we study an application of a similar model to basket credit derivatives.

We mention in passing that equations of the form (\ref{spde}) arise also in stochastic filtering. To be precise, (\ref{spde}) is the Zakai equation for the distribution of a signal $X$ given observation of $M$, see e.g.\ \cite{crisan}.

It is interesting to note that the solution to the SPDE (\ref{spde}) without boundary conditions can be written as the solution of the PDE
\begin{equation}
\frac{\partial u}{\partial t} = \frac{1}{2} (1\!-\!\rho)\, \frac{\partial^2 u}{\partial x^2}
 - \mu\, \frac{\partial u}{\partial x},
\label{pde}
\end{equation}
shifted by the current value of the Brownian driver,
\begin{equation}
\label{shift}
v(t,x) = u(t, x\!-\!\sqrt{\rho}\,M_t).
\end{equation}
In particular, if $v(0,x)=\delta(x_0\!-\!x)$, then
\begin{equation}
v(T,x) = \frac{1}{\sqrt{2\pi \, (1\!-\!\rho)\,T}}\ 
\exp\left( -\ \frac{(\rule{0in}{0.16in}x - x_0 - \mu\, T - \sqrt{\rho} \, M_T)^2}{2\,(1\!-\!\rho)\,T} \right).
\label{eq:gaussian}
\end{equation}

The intuitive interpretation of this result is that the independent Brownian motions $W^i$ have averaged into a deterministic diffusion in the infinite particle limit, whereas the common factor $M$, which moves all processes in parallel, shifts the
whole profile (and also adds to the diffusion, via the It\^{o} term).

In \cite{bush}, the analysis of the large particle system is extended to cases with absorption at the boundary ($x=0$),
\begin{eqnarray}
\nonumber
X_t^i &=& 0,\quad t\ge T_0^i, \\
T_0^i &=& \inf \{t: X_t^i=0\}.
\label{defaulttime}
\end{eqnarray}
It is shown
that there is still a limit measure $\nu_t$, which may now be decomposed as
\[
\nu_t = \nu_t^+ + L_t \delta_0, 
\]
where
\begin{equation}
\label{loss-var}
L_t = \lim_{N\rightarrow \infty} L_t^N =
\lim_{N\rightarrow \infty} \frac{1}{N} \sum_{i=1}^N 1_{T_0^i\le t}
\end{equation}
is the proportion of absorbed particles (the ``loss function''), and
the density $v$ of $\nu_t^+$ satisfies (\ref{spde}) in $(0,\infty)$ with
absorbing boundary condition
\begin{equation}
\label{bc}
v(t,0) = 0.
\end{equation}

\cite{bush} consider applications to basket credit derivatives.
For the market pricing examples, they consider a simplified model, where
defaults are monitored only at a discrete set of dates.
Between these times, the default barrier is inactive and (\ref{spde}) is solved on the real line by using (\ref{pde}) and (\ref{shift}).

For the initial-boundary value problem (\ref{spde}), (\ref{bc}), however, such a semi-analytic solution strategy is no longer possible and an efficient numerical method is needed.
Moreover, there is a loss of regularity at the boundary in this case, 
such that $x u_{xx} \in L_2$ but not $u_{xx}$, as is documented in \cite{krylov94}.

Recent papers on the numerical solution of SPDEs deal with cases relevant to ours, yet structurally crucially different.
A comprehensive analysis of finite difference and finite element discretisations of the stochastic heat equation with multiplicative white noise and non-linear driving term is given in \cite{gyongy97, gyongy} and \cite{walsh}, respectively.
\cite{lang10} shows a Lax equivalence theorem for the SDE
\[
{\rm d} X_t = A X_t \dt + G(X_t) \dM_t,
\]
in a Hilbert space, driven by a process $M$ from a class including Brownian motion, where $A$ is a suitable (e.g.\ elliptic differential) operator and $G$ a Lipschitz function.


In this paper, we propose a Milstein finite difference discretisation for (\ref{spde}) and analyse its stability and convergence in the mean-square sense by Fourier analysis. 
A main consideration of this paper is the computational complexity of the proposed methods, and we will
demonstrate that a multilevel approach achieves a cost for the SPDE simulation no larger than that  of direct Monte Carlo sampling from a known univariate distribution, $O(\eps^{-2})$ for r.m.s.~accuracy $\epsilon$, and is in that sense optimal.

Multilevel Monte Carlo path simulation, first introduced in \cite{giles1}, 
is an efficient technique for computing expected values of path-dependent payoffs
arising from the solution of SDEs.  It is based on a multilevel decomposition 
of Brownian paths, similar to a Brownian Bridge construction. The complexity 
gain can be explained by the observation that the variance of high-level 
corrections -- involving a large number of timesteps -- is typically small, 
and consequently only a relatively small number Monte Carlo samples is required 
to estimate these contributions to an acceptable accuracy.  Overall, for SDEs,
if a r.m.s.~accuracy of $\eps$ is required, the standard Monte Carlo method 
requires $O(\eps^{-3})$ operations, whereas the multilevel method based on the 
Milstein discretisation \cite{giles2} requires $O(\eps^{-2})$ operations.


The first extension of the multilevel approach to SPDEs was for parabolic 
PDEs with a multiplicative noise term \cite{graubner}.
There have also been recent extensions to 
elliptic PDEs with random coefficients \cite{barth,cliffe}.


Our approach, for a rather different parabolic SPDE, is similar to the
previous work on SDEs and SPDEs in that the solution is decomposed 
into a hierarchy with increasing resolution in both time and space.
Provided the variance 
of the multilevel corrections decreases at a sufficiently high rate as one moves to higher levels of refinement,
the number of fine grid Monte Carlo simulations which is required is greatly reduced.
Indeed, the total cost is only 
$O(\eps^{-2})$ to achieve a r.m.s.~accuracy of 
$\eps$ compared to an $O(\eps^{-7/2})$ cost for the standard approach 
which combines a finite difference discretisation of the spatial derivative terms 
and a Milstein discretisation of the stochastic integrals.

The rest of the paper is structured as follows. 
Section \ref{sec:finite-difference} outlines the finite difference scheme
used, and analyses its accuracy and stability in the standard Monte Carlo approach.
Section \ref{sec:multilevel} presents the modification to multilevel path simulation of functionals of the solution.
Numerical experiments for a CDO tranche pricing application
are given in section \ref{sec:results}, providing 
empirical support for the postulated properties of the scheme and demonstrating
the computational gains achieved.
Section \ref{sec:conclusions} discusses the benefits over standard Monte Carlo simulation
of particle systems and outlines extensions.

\section{Discretisation and convergence analysis}
\label{sec:finite-difference}

\subsection{Milstein finite differences}

Integrating (\ref{spde}) over the time interval $[t,t\!+\!k]$ gives
\[
v(t\!+\!k,x) = v(t,x) + \int_t^{t+k} \left( -\,  \mu \frac{\partial v}{\partial x}
                + \frac{1}{2}\, \frac{\partial^2  v }{\partial x^2}\right) \ds - 
          \int_t^{t+k} \!\!\sqrt{\rho}\ \frac{\partial v}{\partial x}\,  \dM_s.
\]
Making the approximation $v(s,x) \!\approx\! v(t,x)$ for $t\!<\!s\!<\!t\!+\!k$ 
in the first integral and
\[
v(s,x) \approx v(t,x) - \sqrt{\rho} \left.\frac{\partial v}{\partial x}\right|_{(t,x)} (M_s\!-\!M_t)
\]
in the second, and noting the standard \Ito calculus result that 
\[
\int_t^{t+k} (M_s\!-\!M_t)\,  \dM_s = \frac{1}{2} \left(\rule{0in}{0.16in}(\Delta M^n)^2 \!-\! k\right),
\]
where $\Delta M^n \equiv M_{t+k}-M_t = \sqrt{k}\, Z_n$ with $Z_n \sim N(0,1)$, \cite{glasserman,kp92},
leads to the Milstein semi-discrete approximation
\[
v^{n+1}(x) = v^n(x) - (\mu\, k +\sqrt{\rho\, k}\, Z_n)\, \frac{\partial v^n}{\partial x} 
                + \frac{1}{2} \left(\rule{0in}{0.16in} (1\!-\!\rho)\, k + \rho \, k\, Z_n^2\right)
                 \frac{\partial^2  v^n }{\partial x^2}.
\]
Using a spatial grid with uniform spacing $h$, standard central difference approximations 
to the spatial derivatives \cite{rm67} then give the finite difference equation
\begin{eqnarray}
v_j^{n+1} &=& v_j^n\ -\ \frac{\mu\, k + \sqrt{\rho\, k}\, Z_n}{2h} \left(v_{j+1}^n - v_{j-1}^n\right) 
\nonumber \\&&~~~~ +\ \frac{(1\!-\!\rho)\, k + \rho \, k\, Z_n^2}{2h^2} 
 \left(v_{j+1}^n - 2 v_j^n + v_{j-1}^n\right),
\label{discrete}
\end{eqnarray}
in which $v_j^n$ is an approximation to $v(nk,jh)$. 

The spatial domain is truncated by introducing 
an upper boundary at $x_{max} \!=\! J\,h$ and using the boundary condition $v_J^n\!=\!0$.  
Since the initial distribution will be assumed localised,  both the localisation error for a given path 
$M$, and the expected error of functionals of the solution, can be made as small as needed by a large enough choice $x_{max}>0$. 

The system of SDEs can then be written in matrix-vector form
\begin{eqnarray}
V_{n+1} &=& V_n\ -\ \frac{\mu\, k + \sqrt{\rho\, k}\, Z_n}{2h} D_1 V_n
+\ \frac{(1\!-\!\rho)\, k + \rho \, k\, Z_n^2}{2 h^2}  D_2 V_n,
\label{discrete-sys}
\end{eqnarray}
where $V_n$ is the vector with elements $v_j^n, j=1, \ldots, J\!-\!1$ and $D_1$ and $D_2$ are
the matrices corresponding to first and second central differences as explicitly given in
Appendix \ref{sec:matanal}.

\begin{remark}
\label{remark-mol}
An alternative discretisation arises if the spatial derivatives are discretised first, and the Milstein scheme is applied to the resulting system of SDEs. The practical implication is that the \Ito term then contains a second finite difference with twice the step size, $D_1^2$ instead of $D_2$, resulting in pentadiagonal discretisation matrices instead of tridiagonal ones, specifically
\begin{eqnarray}
V_{n+1} = V_n - \frac{\mu\, k + \sqrt{\rho\, k}\, Z_n}{2h} D_1 V_n
+  \frac{k}{2 h^2}  D_2 V_n
+  \frac{(\rho \, k\, (Z_n^2-1)}{2 h^2}  D_1^2 V_n.
\label{discrete-sys2}
\end{eqnarray}
The accuracy and stability of the two schemes are similar, hence we will not go into details. 
\end{remark}

To approximate the initial condition on the grid, we apply the initial
measure $\nu^N_0$ to a basis of `hat functions'
$\left\langle\Psi_j\right\rangle_{0<j<J}$, where
\[
\Psi_j(x) = \frac{1}{h} 
\max \left( h-|x-x_j|,0 \right),
\]
giving
\[
v_j^0 = \langle \Psi_j, v_0 \rangle =
\int_{-\infty}^{\infty}
\Psi_j(x)\ v_0(x) {\, \rm d}x.
\]
For the corresponding PDE ($\rho=0$), this is known to result in $O(h^2)$ 
convergence provided $k$ satisfies a certain stability limit, even when the
initial data are not smooth
\cite{pooley1, carter}.
We will see that an extension of this analysis holds for SPDEs.

\subsection{Fourier stability analysis}

Finite difference Fourier stability analysis ignores the boundary conditions and considers
Fourier modes of the form
\begin{equation}
\label{fouriermode}
v_j^n = g_n \exp(i j \theta), \quad |\theta| \leq \pi,
\end{equation}
which satisfy (\ref{discrete}) provided
\[
g_{n+1} = \left( a(\theta) + b(\theta)\, Z_n + c(\theta)\, Z_n^2 \right) g_n,
\]
where
\begin{eqnarray*}
a(\theta) &=& 1 - \frac{i \, \mu\, k}{h}\, \sin \theta
 - \frac{2\, (1\!-\!\rho)\, k}{h^2} \sin^2 \fracs{\theta}{2}, \\[0.05in]
b(\theta) &=& -\, \frac{i \sqrt{\rho\, k}}{h}\, \sin \theta, \\[0.05in]
c(\theta) &=& -\, \frac{2\, \rho\, k}{h^2} \sin^2 \fracs{\theta}{2}.
\end{eqnarray*}
Following the approach of mean-square stability analysis from \cite{higham, sm96}, we obtain
\begin{eqnarray*}
\EE[\, |g_{n+1}|^2 ] 
&=& \EE\left[ (a + b\, Z_n + c\, Z_n^2)(a^* + b^* Z_n + c^* Z_n^2)\  |g_n|^2 \right] \\[0.1in]
&=& \left( \, |a\!+\!c|^2 + |b|^2 + 2 |c|^2 \,\right)\ \EE\left[\, |g_n|^2 \right],
\end{eqnarray*}
where $a^*$ denotes the complex conjugate of $a$.  Mean-square stability therefore requires
\begin{eqnarray*}
&& |a|^2 + |b|^2 + 3 |c|^2 + a\, c^* + a^* c \\
&& =\ 1
 - 4 \sin^2 \fracs{\theta}{2} \left\{ \frac{k}{h^2} 
- (1 + 2 \rho^2) \left(\frac{k}{h^2}\right)^2 \sin^2\fracs{\theta}{2}
-  \left( \left(\frac{\mu\,k}{h}\right)^2 + \frac{\rho\, k}{h^2} \right) \cos^2\fracs{\theta}{2}
\right\} \\
&& \leq \ 1,
\end{eqnarray*}
and enforcing this for all $\theta$ leads to 
the two conditions summarised in the following theorem.

\begin{theorem}
The Milstein finite difference scheme (\ref{discrete-sys}) is stable in the mean-square sense provided
\begin{eqnarray}
\mu^2 k &\leq& 1 - \rho, \\
\frac{k}{h^2} &\leq& (1+2\rho^2)^{-1}.
\label{expl-stab-lim}
\end{eqnarray}
\end{theorem}

The analysis in Appendix \ref{sec:matanal} combines mean-square and matrix stability analysis
to prove that in the limit $k,h\rightarrow 0$ the condition 
$k/h^2 \leq (1+2\rho^2)^{-1}$ is also a sufficient condition for mean-square stability
of the initial-boundary value problem with the boundary conditions at $j\!=\!0$ and $j\!=\!J$.

\subsection{Fourier analysis of accuracy}
\label{subsec:fourier-acc}

Fourier analysis can also be used to examine the accuracy of the finite difference 
approximation in the absence of boundary conditions.
Considering a Fourier mode of the form
$
g(t) \exp( i \kappa x),
$
the PDE (\ref{pde}) reveals that
\[
g(t) = g(0) \ \exp\left(\rule{0in}{0.16in}\! -\fracs{1}{2} (1\!-\!\rho)\, \kappa^2\, t
        - i \, \kappa\, (\mu\, t + \sqrt{\rho}\, M_t) \right),
\]
and therefore
\[
g(t_{n+1}) = g(t_n) \ \exp\left(\rule{0in}{0.16in}\! -\fracs{1}{2} (1\!-\!\rho)\, \kappa^2 \, k
        - i \,\kappa\, (\mu\, k \!+\! \sqrt{\rho\,k}\, Z_n) \right),
\]
where $M_{t_{n+1}}- M_{t_n} = \sqrt{k}\, Z_n$.
Fourier analysis of the discretisation gives
\[
g_{n+1} = \left( a(\kappa h) + b(\kappa h)\, Z_n + c(\kappa h)\, Z_n^2\right) g_n,
\]
where $a(\theta)$, $b(\theta)$, $c(\theta)$ are as defined before.
Writing
\[
 a(\kappa h) + b(\kappa h)\, Z_n + c(\kappa h)\, Z_n^2 = 
\exp\left(\rule{0in}{0.16in}\! -\fracs{1}{2} (1\!-\!\rho)\, \kappa^2\, t
        - i \, \kappa\, (\mu\, k \!+\! \sqrt{\rho\,k}\, Z_n) + e_n \right)
\]
we obtain, after performing lengthy expansions using MATLAB's Symbolic Toolbox,
\[
e_n = \sqrt{\rho\,  k} \ \kappa^2 \mu\, k\, Z_n
- \fracs{1}{6} i\, \sqrt{\rho\,  k} \ \kappa^3 Z_n
\left(\rule{0in}{0.16in}3\, (1\!-\!\rho)\, k + \rho\, k\, Z_n^2 - h^2\right)
 + r_n,
\]
where, for all $p\ge 1$,
\[
\EE[|r_n|^p] \le c(\kappa,p) \left(k^2 + k h^2\right)^p.
\]
When summing over $T/k$ timesteps, $\sum Z_n$ and $\sum Z^3_n$ are both 
$O(k^{-1/2})$ since they have zero expectation and $O(k^{-1})$ variance.
Hence, it follows that
\[
\left( \EE \left[ \left( \sum_n e_n \right)^2 \right]\right)^{1/2}= O(k, h^2),
\]
so the RMS error in the Fourier mode over the full simulation interval is 
$O(k, h^2)$.


Following the method of analysis in \cite{carter}, it can be deduced from
this that the RMS $L_2$ and $L_\infty$ errors for $V_n$ 
are both $O(k,h^2)$.
This is consistent with the usual $O(k,h^2)$ accuracy of the 
forward-time central space discretisation of a parabolic PDE \cite{rm67}, 
and the $O(k)$ strong accuracy of the Milstein discretisation of an SDE, see e.g.\
\cite{kp92}.

\subsection{Convergence tests}
\label{subsec:firsttests}

We now test the accuracy and stability of the scheme numerically.

The chosen parameters for (\ref{spde}) are
$\sigma = 0.22$, $\rho = 0.2$, $r=0.042$, $\mu=(r-\sigma^2/2)/\sigma$,
$v(0,x)\!=\!\delta(x\!-\!x_0)$ with $x_0 = 5$. These values are typical for the applications later on.

The upper boundary for the computation is $x_{max}=16$, and chosen
to ensure that the use of zero Dirichlet data has negligible influence on the solution.

We approximate the initial value problem without boundary conditions on 
$[-16/3,16]$ (note $x_0=5$ is roughly in the centre), 
and also the initial-boundary value problem with zero Dirichlet conditions on
$[0,16]$.

Figure \ref{fig:solns} plots several solutions to the latter problem at $T=5$,
each corresponding to a different Brownian path $M_t$.
\begin{figure}
\begin{center}
\includegraphics[width=.6\textwidth]{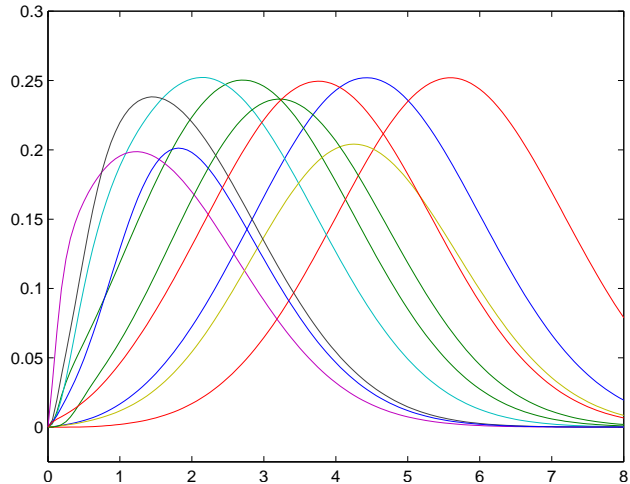}
\end{center}
\caption{Some solutions $v(T,x)$ for different driving 
         Brownian motions $W_t$.}
\label{fig:solns}
\end{figure}
It can be seen that for those realisations for which $M_t$ has been largely 
positive, the boundary at $x\!=\!0$ has had negligible 
influence and so the solution is approximately equal to the 
displaced Normal distribution given by (\ref{eq:gaussian}).

For the unbounded case, we approximate the mean-square $L_2$-error by
\begin{eqnarray}
\nonumber
E(h,k)^2 &=& \EE\left[\sum_{j=0}^J (v_j^N(\omega)-v(N k,j h;\omega))^2 \,  h\right] \\
&\approx& \frac{1}{M} \sum_{m=1}^M
\sum_{j=0}^J (v_j^N(\omega_m)-v(N k,j h;\omega_m))^2 \, h,
\label{errex}
\end{eqnarray}
where $J$ is the number of grid intervals, $N$ the number of timesteps, and
the expectation is taken over Brownian paths $\omega$, of which $\omega_m$
are $M$ samples.

To study the convergence, we introduce decreasing grid sizes 
$h_l = h_0\, 2^{-l}$ and timesteps
$k_l = k_0\, 4^{-l}$, motivated by the second order convergence in $h$
and first order convergence in $k$ as well as the stability limit for the 
explicit scheme, and denote $E_l = E(h_l,k_l)$ the mean-square $L_2$-error
at level $l$.

For the initial-boundary value problem, no analytical solution is known,
but we can compute error indicators via the difference between a fine grid 
solution $f$ with mesh parameters $k$ and $h$, 
and a coarse solution $c$ with mesh parameters $4 k$ and $2 h$,
\begin{eqnarray}
\nonumber
e(h,k)^2 &=& \EE\left[\sum_{j=0}^{J/2} (f_{2j}^{N}(\omega)-c_j^{N/4}(\omega))^2 \, h\right] 
\\
&\approx& \frac{1}{M} \sum_{m=1}^M
\sum_{j=0}^{J/2} (f_{2j}^N(\omega_m)-c_j^{N/4}(\omega_m))^2 \, h,
\label{errest}
\end{eqnarray}
and define $e_l = e(h_l,k_l)$.

On the coarsest level, $h_0=4/3$, $k_0=1/4$, such that $x_0$ does not coincide 
with a grid point.

Fig.~\ref{fig:unbddandbddplots} shows both the computed values of $E_l^2$ and $e_l^2$ for the unbounded case, and $e_l^2$ for the bounded case.
\begin{figure}
\begin{center}
\includegraphics[width=.7\textwidth]{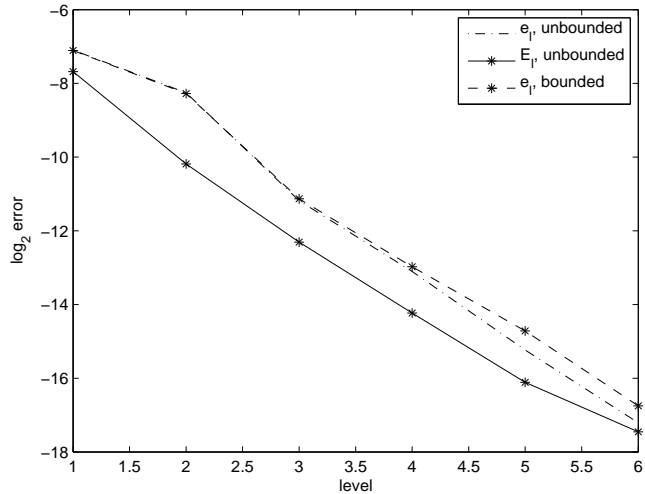}
\end{center}
\caption{
Mean-square error estimators for the
Milstein scheme.
For the unbounded case, $L_2$ difference between true and numerical solution,
$E_L$ as in (\ref{errex}),
and coarse and fine grid solutions, $e_l$ as in (\ref{errest});
for the bounded case, difference between coarse and fine grid solutions
$e_l$.
}
\label{fig:unbddandbddplots}
\end{figure}
The results confirm the theoretical $O(k,h^2)$ convergence in the unbounded case, and support the conjecture that the convergence order is unchanged in
the bounded case.



We now formalise this conjecture about the error due to the SPDE discretisation.
Denote $U_T = u(T,\cdot)$ the solution to the initial boundary value problem at time $T$ for a given Brownian path
and $\hU_T$ its numerical approximation with grid size $h$ and timestep
$k\propto h^2$.
\begin{conjecture}
\label{conj}
The error in the solution at time $T$ satisfies the strong
error estimate in the $L_2$ norm $\|\cdot \|$
\[
\sqrt{ \EE[ \|\hU_T - U_T \|^2 ] }= O(h^{2}).
\]
\end{conjecture}
\begin{corollary}
A Lipschitz payoff function $P(U_T)$ has a similar strong error bound
\[
\sqrt{\EE[ (P(\hU_T) - P(U_T))^2 ]} = O(h^{2}).
\]
\end{corollary}

We know Conjecture \ref{conj} to be true for the initial value problem on $\RR$,
and it conjectures that the introduction of the boundary condition at $x\!=\!0$ 
does not affect the (weak and) strong error,
which is supported by the numerical results. 

The conjecture assumes $x_{max}\!=\!\infty$; in practice, 
a finite value for $x_{max}$ will introduce an additional truncation error 
which will decay 
exponentially in $x_{max}$.

Numerical tests with different parameters, not reproduced here,
indicate that the error in the bounded case 
increases for values of $x_0$ close to 0,
in which case more paths contribute solutions with large higher derivatives 
close to the zero boundary, even if the asymptotic convergence order is still 
$O(k,h^2)$.
We analyse this loss of regularity further in Appendix \ref{subsec:regularity}.

\section{Multilevel Monte Carlo simulation}
\label{sec:multilevel}

We now consider estimating the expectation of a scalar functional of the SPDE
solution, an expected ``payoff'' $P$ in a computational finance context.

Defining $\hP$ to be the payoff arising from a single SPDE approximation, the 
standard Monte Carlo approach is to average the payoff 
from $N$ independent SPDE simulations, each one using an independent vector of 
$N(0,1)$ Normal variables $Z = (Z^0, Z^1, \ldots, Z^{n-1})$.  Thus the estimator 
for the expected value is
\[
\hY = \frac{1}{N} \sum_{i=1}^N \hP^{(i)}.
\]
The mean square error for this estimator can be expressed as the sum of two 
contributions, one due to the variance of the estimator and the other due to the 
error in its expectation,
\[
\EE \left[ \left( \hY - \EE[P] \right)^2 \right] 
= N^{-1} \VV[\hP] + \left( \EE[\hP] \!-\! \EE[P] \right)^2.
\]
To achieve a r.m.s.~error of $\eps$ requires that both of these terms are 
$O(\eps^2)$.  This in turn requires that $N\!=\!O(\eps^{-2})$, $k\!=\!O(\eps)$ 
and $h\!=\!O(\eps^{1/2})$, based on the conjecture that the weak error 
$\EE[\hP] \!-\! \EE[P]$ is $O(k,h^2)$.
Since the computational cost is proportional to $N k^{-1} h^{-1}$ 
this implies an overall cost which is $O(\eps^{-7/2})$.  The aim of the multilevel 
Monte Carlo simulation is to reduce this complexity to $O(\eps^{-2})$. 

Consider Monte Carlo simulations with different levels of refinement,
$l = 0, 1, \ldots, L$, with $l=0$ being the coarsest level, 
(i.e.~the largest values for $k$ and $h$) and level $L$ being the finest
level corresponding to that used by the standard Monte Carlo method.

Let $\hP_l$ denote an approximation to payoff $P$
using a numerical discretisation with parameters $k_l$ and $h_l$.  
Because of the linearity of the expectation operator, it is clearly 
true that
\begin{equation}
\EE[\hP_L] = \EE[\hP_0] + \sum_{l=1}^L \EE[\hP_l \!-\! \hP_{l-1}].
\label{eq:identity}
\end{equation}
This expresses the expectation on the finest level as being equal to
the expectation on the coarsest level plus a sum of corrections which
give the difference in expectation between simulations using 
different numbers of timesteps.
The multilevel idea is to independently estimate each of the expectations 
on the right-hand side in a way which minimises the overall variance for 
a given computational cost.

Let $\hY_0$ be an estimator for $\EE[\hP_0]$ using $N_0$ samples, and let 
$\hY_l$ for $l\!>\!0$ be an estimator for $\EE[\hP_l \!-\! \hP_{l-1}]$ 
using $N_l$ samples.  Each estimator is an average of $N_l$ independent 
samples, which for $l\!>\!0$ is
\begin{equation}
\hY_l = N_l^{-1} \sum_{i=1}^{N_l} \left( \hP_l^{(i)} \!-\! \hP_{l-1}^{(i)} \right).
\label{eq:est_l}
\end{equation}
The key point here is that the quantity $\hP_l^{(i)} \!-\! \hP_{l-1}^{(i)}$ comes 
from two discrete approximations using the same Brownian path.  
The variance of this simple estimator is 
$\displaystyle
\VV[\hY_l] = N_l^{-1} V_l
$
where $V_l$ is the variance of a single sample.  
Combining this with independent estimators for each of the other levels,
the variance of the combined estimator
$
\sum_{l=0}^L \hY_l
$
is
$
\sum_{l=0}^L N_l^{-1} V_l.
$
The corresponding computational cost is
$
\sum_{l=0}^L N_l\, C_l
$
where $C_l$ represents the cost of a single sample on level $l$.
Treating the $N_l$ as continuous variables, the variance is minimised 
for a fixed computational cost by choosing $N_l$ to be proportional to
$\displaystyle \sqrt{V_l / C_l}$, with the constant of proportionality 
chosen so that the overall variance is $O(\eps^{-2})$.

The total cost on level $l$ is proportional to $\sqrt{V_l \, C_l}$.  
If the variance $V_l$ decays more rapidly with level than the cost $C_l$
increases, the dominant cost is on level $0$.  The number of samples on
that level will be $O(\eps^{-2})$ and the cost savings compared to standard
Monte Carlo will be approximately $C_0/C_L$, reflecting the different costs
of samples on level $0$ compared to level $L$.  On the other hand, if the
variance $V_l$ decays more slowly than the cost $C_l$ increases, the dominant 
cost will be on the finest level $L$, and the cost savings compared to 
standard Monte Carlo will be approximately $V_L/V_0$, reflecting the 
difference between the variance of the finest grid correction compared to 
the variance of the standard Monte Carlo estimator, which is similar to $V_0$.

This outline analysis is made more precise in the following theorem:


\begin{theorem}
\label{ml-theorem}
Let $P$ denote a functional of the solution of an SPDE for a given 
Brownian path $M_t$, and let $\hP_l$ denote the corresponding level $l$ 
numerical approximation.

If there exist independent estimators $\hY_l$ based on $N_l$ Monte Carlo 
samples, and positive constants 
$\alpha, \beta, \gamma, c_1, c_2, c_3$ such that 
$\alpha\!\geq\!\fracs{1}{2}\,\gamma$ and
\begin{enumerate}[label=\roman{*})]
\item
$\displaystyle
\left| \EE[\hP_l \!-\! P] \right| \leq c_1\, 2^{-\alpha\, l}
$
\item
$\displaystyle
\EE[\hY_l] = \left\{ \begin{array}{ll}
\EE[\hP_0], & l=0 \\[0.1in]
\EE[\hP_l \!-\! \hP_{l-1}], & l>0
\end{array}\right.
$
\item
\label{ml-iii}
$\displaystyle
\VV[\hY_l] \leq c_2\, N_l^{-1} 2^{-\beta\, l}
$
\item
\label{ml-iv}
$\displaystyle
C_l \leq c_3\, N_l\, 2^{\gamma\, l},
$
where $C_l$ is the computational complexity of $\hY_l$
\end{enumerate}
then there exists a positive constant $c_4$ such that for any $\eps \!<\! e^{-1}$
there are values $L$ and $N_l$ for which the multilevel estimator
\[
\hY = \sum_{l=0}^L \hY_l,
\]
has a mean-square-error with bound
\[
MSE \equiv \EE\left[ \left(\hY - E[P]\right)^2\right] < \eps^2
\]
with a computational complexity $C$ with bound
\[
C \leq \left\{\begin{array}{ll}
c_4\, \eps^{-2}              ,    & \beta>\gamma, \\[0.1in]
c_4\, \eps^{-2} (\log \eps)^2,    & \beta=\gamma, \\[0.1in]
c_4\, \eps^{-2-(\gamma\!-\!\beta)/\alpha}, & 0<\beta<\gamma.
\end{array}\right.
\]
\end{theorem}

\begin{proof}
The proof is a slight generalisation of the proof in \cite{giles1}.
\end{proof}

In our application, we choose $h_l \propto 2^{-l}$ and $k_l\propto 4^{-l}$
so that the ratio $k_l/h_l^2$ is held fixed to satisfy the finite
difference stability condition.
The computational cost increases by factor 8 in moving from level $l$ to 
$l\!+\!1$, so $\gamma\!=\!3$.

Given that the payoff is a Lipschitz function of the loss
approximations at various dates, Conjecture \ref{conj} implies that the weak 
error is also $O(h^2)$ and so $\alpha \!=\! 2 > \fracs{1}{2}\,\gamma$.
Also, due to the triangle inequality
\begin{eqnarray*}
\sqrt{\VV[ \hP_l \!-\! \hP_{l-1} ]} 
&\leq& 
\sqrt{\VV[ \hP_l \!-\! P ]} +
\sqrt{\VV[ \hP_{l-1} \!-\! P ]} 
\\ &\leq& 
\sqrt{\EE[ (\hP_l \!-\! P)^2 ]} +
\sqrt{\EE[ (\hP_{l-1} \!-\! P)^2 ]} ,
\end{eqnarray*}
Conjecture \ref{conj} and its corollary give $\beta \!=\! 4$.
Consequently, the computational cost to achieve a r.m.s.~error of $\eps$
is $O(\eps^{-2})$.

\section{Numerical experiments}
\label{sec:results}

In this section we study the numerical performance of the algorithms presented earlier on the example of CDO tranche pricing in a large basket limit.

\subsection{CDO pricing in the structural credit model}
\label{sec:model}

Basket credit derivatives provide protection against the default of a 
certain segment (`tranche') of a basket of firms. A typical example is that of a 
\emph{collateralized debt obligation} where the protection buyer 
receives a notional amount, minus some recovery proportion $0\le R\le 1$,
if firms in a specified tranche of the basket 
default, and in return pays a regular spread until the default event 
occurs.

The arbitrage-free spread depends crucially on the
(risk-neutral, when hedged with defaultable bonds) probability of joint 
defaults. 
For a tranche with attachment point $0\le a < 1$ and detachment point $1\ge d>a$, the outstanding tranche notional
\begin{equation}
\label{tranche-notional}
P(L_t) = \max(d-L_t,0) - \max(a-L_t,0),
\end{equation}
where the loss variable $L_t$ is the proportion of losses in the basket up to time $t$,
determines the spread and default payments related to that tranche.
 The risk-neutral value of the tranche spread can be derived as
\begin{equation}
\label{tranche-spreads}
s = \frac{\sum_{i=1}^{n}
\e^{-r T_{i}} \mathbb{E}^{\mathbb{Q}}[P(L_{T_{i-1}}) - P(L_{T_{i}})]}
{\delta \sum_{i=1}^{n} \e^{-r T_i} \mathbb{E}^{\mathbb{Q}}[P(L_{T_{i}})]},
\end{equation}
 see, e.g., \cite{bush}. Here,
$n$ is the maximum number of spread payments, $T$ the expiry,
$T_i$ the payment dates for $1\le i\le n$,
$\delta = 0.25$ the interval between payments.
Spreads are quoted as an annual payment, as a ratio of the notional, but assumed to be paid quarterly.
\footnote{There is a variation for the equity tranche, and recently sometimes the mezzanine tranche, but we do not go into details here.}

For an extensive survey of credit derivatives and pricing models we refer the
reader to \cite{schoenbucher}, and note only that they typically fall in 
one of two classes:
so-called reduced-form models, which model default times of firms directly as
(dependent) random variables; structural models, which capture
the evolution of the firms' values, and model default events as the first 
passage of a lower default barrier. We will consider the latter here.

In a structural model in the spirit of the classical works of \cite{merton} and \cite{black},
as extended to multiple firms e.g.\ in the work of \cite{hull},
a company $i$'s firm value, $1\le i\le N$, is assumed to follow a model of the type
\[
\frac{\dA_t^i}{A_t^i} = r \dt + \sqrt{1 \!-\! \rho_i}\ \sigma_i \dW_t^i + \sqrt{\rho_i}\ \sigma_i
\dM_t, \qquad A_0^i = a^i,
\]
where $0\le \rho_i\le 1$ is a correlation parameter,
$r$ is the risk-free interest rate, $\sigma_i$ the volatility, and $M$, $W^i$  are standard Brownian motions.

Here, the individual firms are correlated through a common `market' factor $M$, but independent conditionally on $M$.
We make in the following the assumption that the firms are \emph{exchangeable} in the sense that their dynamics is governed by the same set of parameters, specifically
$\rho=\rho_i$, $\sigma=\sigma_i$. Their initial values $a_i$ are not necessarily identical which
allows for different default probabilities for individual firms, consistent with their CDS spreads.

In this framework, the default time $T_0^i$ of the $i$-th firm is modelled as the first hitting time of a default barrier $B^i$, for simplicity constant, and
the \emph{distance-to-default}
\begin{equation}
\label{distance-to-default}
X_t^i = \frac{1}{\sigma}\left(\,\log A_t^i - \log B^i \,\right),
\end{equation}
evolves according to (\ref{sdesys}), where
$\mu=\left(r- \frac{1}{2} \sigma^2 \right)/\sigma$,
$x^i= \left(\,\log a^i - \log B^i \,\right)/\sigma$.
$T^i_0$ as in (\ref{defaulttime}) is precisely the default time.


In the majority of applications, Monte Carlo simulation of the 
firm value processes is used for the estimation of tranche spreads,
see e.g.\ \cite{hull} or \cite{fouque, carmona}.
This is largely enforced by the size of $N$,
for instance in the case of index tranches $N=125$.

This, however, puts the model precisely in the realm of the large basket approximation
(\ref{spde}).
See \cite{bujok} for a numerical study of this large basket approximation.
%
%
The loss functional is thereby approximated by
\[
\hL_{T_i} = (1-R) \left(1- h \sum_{j=1}^{J-1} v_j^{T_i/k}\right)
\]
in terms of the numerical SPDE solution $v_j^{T_i/k}$ at time $T_i$,
which feeds into the estimator for the outstanding tranche notional (\ref{tranche-notional})
and subsequently the tranche spreads (\ref{tranche-spreads}).


%

\subsection{Pricing results}
\label{subsec:pricing}

All following results are for a representative set of parameter values, taken from a calibration performed in \cite{bush} to 2007 market data,
$\sigma = 0.22$, $\rho = 0.2$, $r=0.042$, $\mu=(r-\sigma^2/2)/\sigma$.
While the initial distribution used in \cite{bush} is somewhat spread out to match individual
CDS spreads of obligors, most of the mass is around $x_0 = 5$ and for
simplicity we centre all firms there for the numerical tests, $v(0,x)\!=\!\delta(x\!-\!x_0)$.
The upper boundary is chosen as $x_{max}=16$. 

We consider a maturity  $T=5$ and tranches
$[a,d] = [0,0.03]$, $[0.03,0.06]$, $[0.06,0.09]$, $[0.09,0.12]$, $[0.12,0.22]$, $[0.22,1]$.

Figure \ref{fig:spde_results} shows the multilevel results for the 
expected protection payment from the first tranche,
\[
\sum_{i=1}^{n} \e^{-r T_{i}} \mathbb{E}^{\mathbb{Q}}[P(\hL_{T_{i-1}}) - P(\hL_{T_{i}})],
\]
expressed as a fraction of the initial tranche notional.

We pick mesh sizes $h_l = h_0 \, 2^{-l}$ and $k_l = k_0 \, 4^{-l}$ for $l>0$ and $h_0=8/5$, $k_0=1/4$.
This is motivated by the second and first order consistency in $h$ and $k$ respectively, and within the stability limit (\ref{expl-stab-lim}).

\begin{figure}
\begin{center}
\includegraphics[width=0.95\textwidth,height=\textwidth]{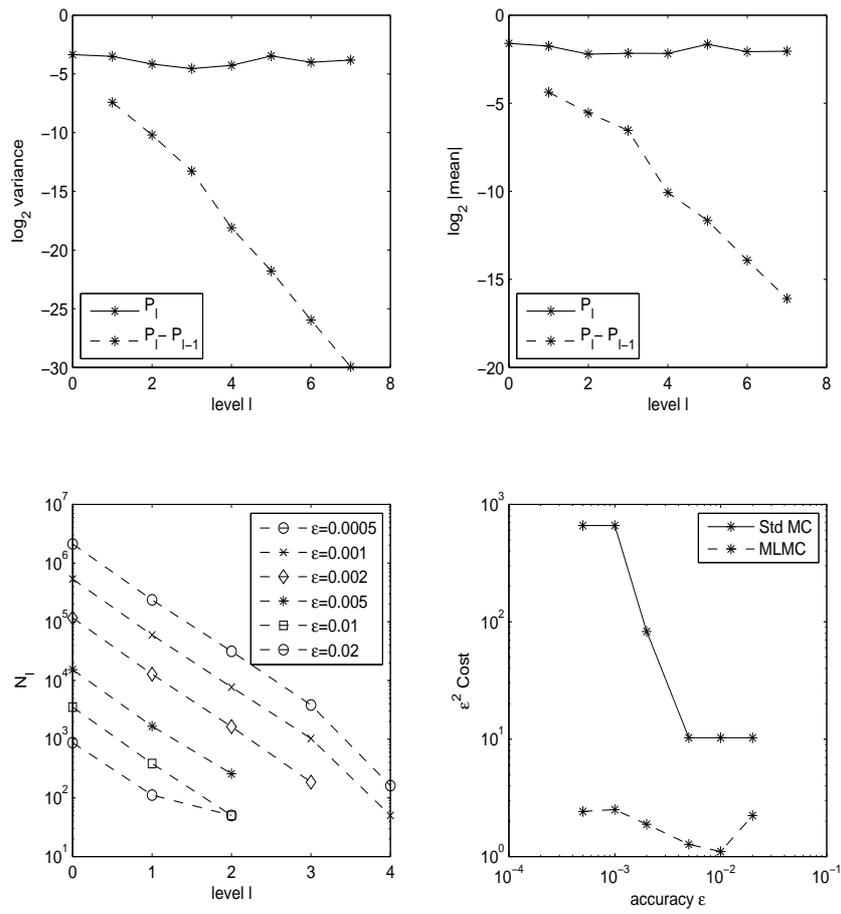}
\end{center}
\caption{Multilevel results for the expected loss from the first tranche.}
\label{fig:spde_results}
\end{figure}

The top-left plot shows the 
convergence of the variance $\VV[\hP_l \!-\! \hP_{l-1}]$ as well as the
variance of the standard single level estimate, while the top-right 
plot shows the convergence of the expectation $\EE[\hP_l \!-\! \hP_{l-1}]$.

The bottom-left plot shows results for different multilevel calculations,
each with a different user-specified accuracy requirement.  For each value 
of the accuracy $\eps$, the multilevel algorithm determines
(by formula (12) from \cite{giles1}) the numbers of
levels of refinement which are needed to ensure that the contribution to 
the Mean Square Error (MSE) due to the weak error on the finest grid is less
than $\fracs{1}{2} \eps^2$, and it determines (by formula (10) from \cite{giles1})
the optimal number of
samples on each level so that the combined variance of the multilevel 
estimate is also less than $\fracs{1}{2} \eps^2$, and hence the MSE is 
less than $\eps^2$.

Treating a single finite difference calculation on the coarsest level 
as having a unit cost, the bottom-right plot compares the total cost
of the standard and multilevel algorithms.  Since the objective is to 
achieve a computational cost which is approximately proportional to 
$\eps^{-2}$, it is the cost $C$ multiplied by $\eps^2$ which is plotted 
versus $\eps$.  The results confirm that $\eps^2 C$ does not vary much
as $\eps\rightarrow 0$ for the multilevel method, in line with the 
prediction that 
$\eps^2 C \propto 1$, whereas it
grows significantly for the standard method since
$\eps^2 C \propto 8^L$, where $L$ is the index of the finest level.
Hence, there is a big jump in the cost of the standard method each time 
it is necessary to switch to a finer level to ensure the weak error 
is less than $\eps/\sqrt{2}$, whereas the jump is minimal
for the multilevel method.




In practice, default is not monitored 
continuously but only at a discrete set of times, and for pricing often the
simplifying assumption is made that default is only determined at the spread
payment dates $T_l$.

This can be incorporated as follows.
In the time intervals $(T_{l-1},T_l)$,
we solve (\ref{spde}) on a sufficiently large domain (e.g.\ $[-4,16]$),
and apply the following interface conditions at $T_l$:
\begin{equation}
\label{discont}
\lim_{t\downarrow T_l} v(t,x) = 
\left\{
\begin{array}{rl}
0, & x\le 0, \\
\lim_{t\uparrow T_l} v(t,x), & x > 0.
\end{array}
\right.
\end{equation}

To maintain quadratic grid convergence in spite of the discontinuity 
introduced by (\ref{discont}), we choose the mesh such that a grid point
coincides with the 0 boundary (e.g.\ by setting $h_0=2$ in the above example),
and set the numerical solution after default monitoring
to 0 for grid coordinates below 0 and to
$1/2$ its previous value at 0, see e.g.~\cite{pooley1}.

It is seen in Fig.~\ref{fig:spde_results_discr} that convergence is very similar
to the previous case.
\begin{figure}
\begin{center}
\includegraphics[width=0.95\textwidth,height=\textwidth]{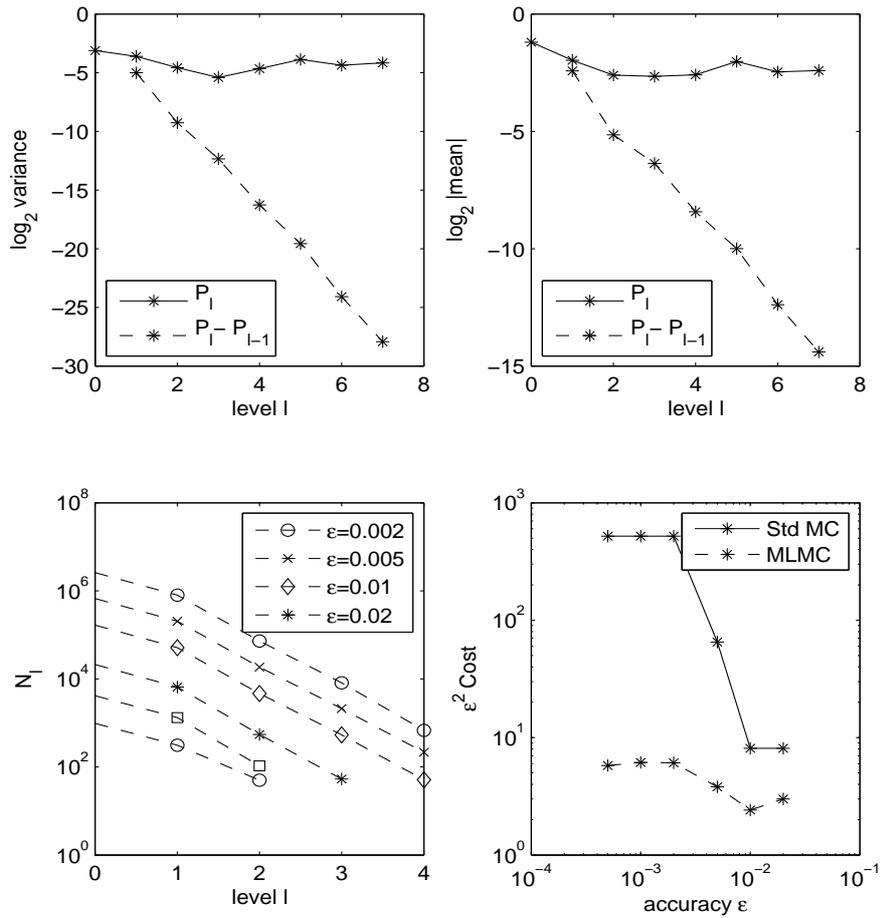}
\end{center}
\caption{Multilevel results for the expected loss from the first tranche,
with discrete monitoring.}
\label{fig:spde_results_discr}
\end{figure}

%
%
%
%
%
%
%
%
%
%

\section{Conclusions}
\label{sec:conclusions}

\subsection{Complexity and cost}
\label{sec:cost}

Here we discuss the computational complexity of the multilevel solution 
of the SPDE compared with the alternative use of the multilevel method for 
solving the SDEs which arise from directly simulating a large number of 
SDEs.

We have already explained that to achieve a r.m.s.~accuracy of $\eps$
requires $O(\eps^{-7/2})$ work when solving the SPDE by the standard 
Monte Carlo method, but the cost is $O(\eps^{-2})$ using
the multilevel method, provided Conjecture \ref{conj} is correct.

Consider now the alternative of using a finite number of particles (firms), $M$,
to estimate the tranche loss in the limit of an infinite number of particles
(firms).
In this case, empirical results suggest that there is an
additional $O(M^{-1})$ error (see also \cite{bujok}), and
the proof of this convergence order is the subject of current research.
Taking this to be the case,
the optimal 
choice of $M$ to minimise the computational complexity to achieve an r.m.s.~error 
of $\eps$ is $O(\eps^{-1})$.  Using the standard Monte Carlo method, 
the optimal timestep is $O(\eps)$, and the optimal number of paths is 
$O(\eps^{-2})$, so the overall cost is $O(\eps^{-4})$.
Using the multilevel method for the SDEs reduces the cost per 
company to $O(\eps^{-2})$, so the total cost is $O(\eps^{-3})$.

This complexity information is summarised in Table \ref{table:comparison}.
There is also a practical implementation aspect to note.  The computational 
cost per grid point in the finite difference approximation of the SPDE is
minimal, requiring just three floating point multiply-add operations if 
equation (\ref{discrete}) is re-cast as
\[
v_j^{n+1} = a v_{j-1}^n + b v_{j}^n + c v_{j+1}^n 
\]
with the coefficients $a, b, c$ computed once per timestep, for all $j$.

If we let $C$ be the cost of generating all of the Gaussian random numbers 
$Z_n$ for a single SPDE simulation,
then the cost of the rest of the finite difference calculation with 20 points
in $x$ (as used on the coarsest level of our multilevel calculations)
is probably similar, giving a total cost of $2C$ for each SPDE.
On the other hand, each SDE needs its own Gaussian random numbers for
the idiosyncratic risk, and so the cost of simulating \emph{each} SDE is 
approximately $C$, roughly half of the cost of the SPDEs on the coarsest 
level of approximation, giving a total cost of $M \, C$.

\begin{table}
\caption{Comparison of the complexity of the SDE and SPDE models using 
the standard and multilevel Monte Carlo methods}
\label{table:comparison}
\begin{center}
\begin{tabular}{|c|c|c|}\hline
\rule{0in}{0.16in} Method / model &   SDE          &    SPDE                   \\ \hline
\rule{0in}{0.16in} Standard MC    & $O(\eps^{-4})$ & $O(\eps^{-7/2})$          \\ \hline
\rule{0in}{0.16in} Multilevel MC  & $O(\eps^{-3})$ & $O(\eps^{-2})$\\ \hline
\end{tabular}
\end{center}

\end{table}

\subsection{Further work}

We have shown that stochastic finite differences combined with a multilevel simulation approach achieve optimal complexity for the computation of expected payoffs of an SPDE model.
In the case of an absorbing boundary, the complexity estimate is a conjecture in so far it relies on the convergence order of the finite difference scheme, which does not follow from the Fourier analysis of the unbounded case.
The matrix stability analysis in Appendix \ref{sec:matanal} could form part of a rigorous analysis if a Lax equivalence theorem could be proved. In the case of multiplicative white noise this is shown in \cite{lang10}.
One difficulty in the present case of an SPDE with stochastic drift is the loss of regularity towards the boundary, which may be accounted for by weighted Sobolev norms of the solution, but even then it is not clear that convergence of the functionals of interest follows.

There are several  possible extensions of the present basic model
as discussed in \cite{bujok}, ranging from stochastic volatility and
jump-diffusion to contagion models.  The methods developed in this paper
should be of use there also, building for example on multilevel versions
of jump-adapted discretisations for jump-diffusion SDEs \cite{platen,xia}.



\appendix
\section{Mean-square matrix stability analysis}
\label{sec:matanal}

If $V_n$ is the vector with elements $v_j^n, j=1, \ldots, J\!-\!1$ then the
finite difference equation can be expressed as
\begin{eqnarray*}
V_{n+1} &=& (A + B\, Z_n + C\,Z_n^2)\ V_n, \\[0.1in]
A &=& I - \frac{\mu\, k}{2h}\, D_1 + \frac{(1\!-\!\rho)\, k}{2h^2}\, D_2, \\[0.05in]
B &=& -\, \frac{\sqrt{\rho\, k}}{2h}\, D_1, \\[0.05in]
C &=&     \frac{\rho\, k}{2h^2}\, D_2,
\end{eqnarray*}
where $I$ is the identity matrix and $D_1$ and $D_2$ are the matrices corresponding 
to central first and second differences, which for 
$J=6$ are
\[
D_1 =
\left(\begin{array}{rrrrr}
  0&1&&& \\
-1&0&1&& \\
&-1&0&1& \\
&&-1&0&1 \\
&&&-1&\ \ 0
\end{array}\right), \quad
D_2 =
\left(\begin{array}{rrrrr}
 -2&1&&& \\
1&-2&1&& \\
&1&-2&1& \\
&&1&-2&1 \\
&&&1&-2
\end{array}\right).
\]

From the recurrence relation we get
\begin{eqnarray*}
\EE[\, V_{n+1}^T V_{n+1} ] 
&=& \EE\left[ V_n^T (A^T + B^T\, Z_n + C^T\, Z_n^2)(A + B\, Z_n + C\, Z_n^2)\  V_n \right] \\[0.1in]
&=& \EE\left[ V_n^T \left( (A\!+\!C)^T (A\!+\!C) + B^T B + 2\, C^T C \,\right) V_n \right].
\end{eqnarray*}

Noting that $D_1$ is anti-symmetric and $D_2$ is symmetric, and that
\[
D_1 D_2 - D_2 D_1 = E_1 - E_2, \quad
D_1^2 = D_3 + E_1 + E_2,
\]
where $D_3$ corresponds to a central second difference with twice the usual span,
\[
D_3 =
\left(\begin{array}{rrrrr}
 -3&0&1&& \\
0 &-2&0&1& \\
1&0&-2&0&1 \\
&1&0&-2&0 \\
&&1&0&-3
\end{array}\right)
\]
(with the end values of $-3$ being chosen to correspond to 
$V_{-1}\equiv - V_1$ and $V_{J+1}\equiv - V_{J-1}$),
and $E_1$ and $E_2$ are each entirely zero apart from one corner element,
\[
E_1 =
\left(\begin{array}{rrrrr}
2&&&& \\
&&&& \\
&&&& \\
&&&& \\
&&&&
\end{array}\right), \qquad
E_2 =
\left(\begin{array}{rrrrr}
&&&& \\
&&&& \\
&&&& \\
&&&& \\
&&&& 2
\end{array}\right),
\]
then after some lengthy algebra we get
\begin{eqnarray*}
\lefteqn{ \EE\left[ V_n^T \left( (A\!+\!C)^T (A\!+\!C) + B^T B + 2\, C^T C\,\right) V_n \right]} \\
& =& \EE\left[ V_n^T M V_n \right]
 - \left( e_1 + e_2 \right) \EE[(v_1^n)^2]
 - \left( e_1 - e_2 \right) \EE[(v_{J-1}^n)^2],
\end{eqnarray*}
where
\[
M = I - \frac{k}{h^2}\, D_2 + \frac{k^2}{4\,h^4}\, D_2^2
- \left( \frac{\rho k}{4\,h^2} + \frac{\mu^2 k^2}{4\,h^2} \right) D_3,
\]
and
\[
e_1 =  \frac{\rho k}{2\,h^2} + \frac{\mu^2 k^2}{2\,h^2}, \qquad
e_2 = \frac{\mu k^2}{2\,h^3}.
\]

It can be verified that the $m^{th}$ eigenvector of $M$ has elements
$\sin j \theta_m$ for $\theta_m = m\, \pi / J$,
and the associated eigenvalue is 
\[
|a(\theta_m)\!+\!c(\theta_m)|^2 + |b(\theta_m)|^2 + 2 |c(\theta_m)|^2,
\]
where $a(\theta), b(\theta), c(\theta)$ are the same functions as defined in the 
mean-square Fourier analysis.  In addition, in the limit $h, k/h\rightarrow 0$,
$e_1 \!\pm\! e_2 > 0$, and therefore in this limit the Fourier stability condition
\[
\sup_\theta \left\{ |a(\theta)\!+\!c(\theta)|^2 + |b(\theta)|^2 + 2 |c(\theta)|^2 \right\}
\leq 1
\]
is also a sufficient condition for mean-square matrix stability.

\section{Regularity considerations}
\label{subsec:regularity}

Figure \ref{fig:conv} shows the convergence behaviour as the computational
grid is refined.  Level 0 has $h \!=\! 1/4, k \!=\! T/4$; $h$ is reduced by factor 2
and $k$ by factor $4$ in moving to finer levels. 
\begin{figure}
\begin{center}
\includegraphics[width=.8\textwidth]{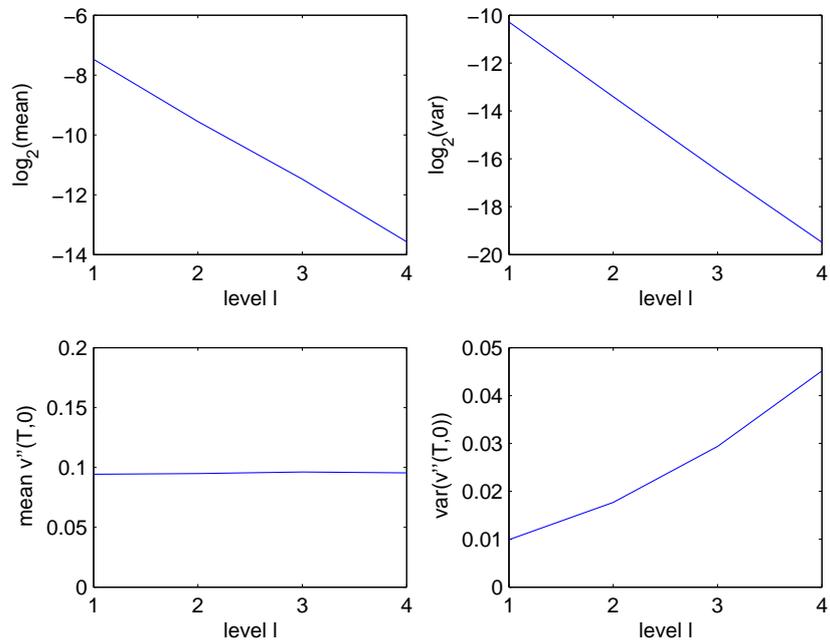}
\end{center}
\caption{Convergence plots as a function of grid level}
\label{fig:conv}
\end{figure}

The top-left plot
shows the convergence of $\EE[\hP_{l}-\hP_{l-1}]$,
with notation as in Sections \ref{sec:multilevel} and \ref{subsec:pricing};
the top-right plot shows the convergence of $\VV[\hP_{l}-\hP_{l-1}]$.

The bottom two plots show the behaviour of $\partial^2 v/\partial x^2(T,0)$,
which is estimated on each grid level using the standard second difference
\[
\frac{\partial^2 v}{\partial x^2}(0) \approx \frac{v_2 - 2v_1 + v_0}{h^2}.
\]
The left plot indicates that the mean of this quantity is well behaved, but the 
right plot indicates a singular behaviour of its variance, with the value 
increasing rapidly with increased grid resolution.

This is in accordance with the result shown in \cite{krylov94}, that for
the unique solution $v$ to the SPDE, $x\, v_{xx}$ is square integrable, but
$v_{xx}$ has a singularity at 0.

\end{document}